\documentclass{amsart}
\title{Incompatibility of generic hugeness principles}
\author{Monroe Eskew}
\address{Universit\"{a}t Wien \\ Institut f\"{u}r Mathematik \\ Kurt G\"odel Research Center \\ Kolingasse 14-16 \\ 1090 Wien, Austria }
\email{monroe.eskew@univie.ac.at}
\date{}

\usepackage{color} 
\usepackage{amssymb} 
\usepackage{amsmath} 
\usepackage[utf8]{inputenc} 
\usepackage{enumerate}
\usepackage{amsthm}
\usepackage{cite}
\usepackage{enumerate}
\usepackage{esint}

\newtheorem{theorem}{Theorem}
\newtheorem{lemma}[theorem]{Lemma}

\newtheorem{corollary}[theorem]{Corollary}

\newtheorem{claim}[theorem]{Claim}

\newtheorem{remark}[theorem]{Remark}

\DeclareMathOperator{\col}{Col}

\DeclareMathOperator{\crit}{crit}

\newcommand{\p}{\mathcal{P}}
\newcommand{\la}{\langle}
\newcommand{\ra}{\rangle}

\newcommand{\bbP}{\mathbb{P}}
\newcommand{\bbQ}{\mathbb{Q}}
\newcommand{\bbR}{\mathbb{R}}

\newcommand{\bbE}{\mathbb{E}}

\newcommand{\calA}{\mathcal{A}}

\newcommand{\calF}{\mathcal{F}}

\newcommand{\calU}{\mathcal{U}}

\begin{document}

\thanks{The author wishes to thank the Austrian Science Fund (FWF) for the generous support through grants P34603 and START Y1012-N35 (PI: Vera Fischer).}
\maketitle

\begin{abstract}
We show that the weakest versions of Foreman's minimal generic hugeness axioms cannot hold simultaneously on adjacent cardinals. Moreover, conventional forcing techniques cannot produce a model of one of these axioms.
\end{abstract}

\section{Introduction}
In \cite{potent, GLC, foremanphil, foremanhandbook}, Foreman proposed generic large cardinals as new axioms for mathematics.  These principles are similar to strong kinds of traditional large cardinal axioms but speak directly about small uncountable objects like $\omega_1,\omega_2$, etc.  Because of this, they are able to answer many classical questions that are not settled by ZFC plus traditional large cardinals.  For example, if $\omega_1$ is minimally generically huge, then the Continuum Hypothesis holds and there is a Suslin line \cite{foremanhandbook}.

For a poset $\bbP$, let us say that a cardinal $\kappa$ is \emph{$\bbP$-generically huge} if $\bbP$ forces that there is an elementary embedding $j : V \to M \subseteq V[G]$ with critical point $\kappa$, where $M$ is a transitive class closed under $j(\kappa)$-sequences from $V[G]$.  If $\bbP$ forces that $j(\kappa) = \lambda$, we call $\lambda$ the \emph{target}.
We say that $\kappa$ is \emph{$\bbP$-generically $n$-huge} when the requirement on $M$ is strengthened to closure under $j^n(\kappa)$-sequences (where $j^n$ is the composition of $j$ with itself $n$ times), and we say $\kappa$ is \emph{$\bbP$-generically almost-huge} if the requirement is weakened to closure under ${<}j(\kappa)$-sequences.
We say that a cardinal $\kappa$ is \emph{$\bbP$-generically measurable} if $\bbP$ forces that there is an elementary embedding $j : V \to M \subseteq V[G]$ with critical point $\kappa$, where $M$ is transitive.

If $\kappa$ is the successor of an infinite cardinal $\mu$, 
we say that $\kappa$ is \emph{minimally generically $n$-huge} if it is $\col(\mu,\kappa)$-generically $n$-huge, where $\col(\mu,\kappa)$ is the poset of functions from initial segments of $\mu$ into $\kappa$ ordered by end-extension.
The main result of this note is that for a successor cardinal $\kappa$, it is inconsistent for both $\kappa$ and $\kappa^+$ to be minimally generically huge.
\begin{theorem}
\label{main}
Suppose $0<m\leq n$ and $\kappa$ is a regular cardinal that is $\bbP$-generically $n$-huge with target $\lambda$, where $\bbP$ is nontrivial and strongly $\lambda$-c.c.  Then $\kappa^{+m}$ is not $\bbQ$-generically measurable for any $\kappa$-closed $\bbQ$.
\end{theorem}

Here, ``nontrivial'' means that forcing with $\bbP$ necessarily adds a new set.
Usuba \cite{usubaapprox} introduced the \emph{strong $\lambda$-chain condition} (strong $\lambda$-c.c.), which means that $\bbP$ has no antichain of size $\lambda$ and forcing with $\bbP$ does not add branches to $\lambda$-Suslin trees.  As Usuba observed, $\bbP$ having the strong $\lambda$-c.c.\ is implied by $\bbP$ having the $\mu$-c.c.\ for $\mu<\lambda$ and by $\bbP \times \bbP$ having the $\lambda$-c.c.  In particular, if $\theta = \kappa^{<\mu}$, then $\col(\mu,\kappa)$ collapses $\theta$ to $\mu$ and is strongly $\theta^+$-c.c. Let us also remark that in Theorem \ref{main}, $\kappa$-closure can be weakened to $\kappa$-strategic-closure without change to the arguments.

Regarding the history: Woodin proved, in unpublished work mentioned in \cite[p.\ 1126]{foremanhandbook}, that it is inconsistent for $\omega_1$ to be minimally generically 3-huge while $\omega_3$ is minimally generically 1-huge.  Subsequently, the author \cite{MR4092254} improved this to show the inconsistency of a successor cardinal $\kappa$ being minimally generically $n$-huge while $\kappa^{+m}$ is minimally generically almost-huge, where $0 < m < n$.  The weakening of the hypothesis to $\kappa$ being only generically 1-huge uses an idea from the author's work with Cox \cite{2020arXiv200914245C}.

In contrast to Theorem \ref{main}, Foreman \cite{foremanmore} exhibited a model where for all $n>0$, $\omega_n$ is $\bbP$-generically almost-huge with target $\omega_{n+1}$ for some $\omega_{n-1}$-closed, strongly $\omega_{n+1}$-c.c.\ poset $\bbP$.  A simplified construction was given by Shioya \cite{MR4159767}.  

We prove Theorem \ref{main} in \S\ref{eva} via a generalization that is less elegant to state.  In \S\ref{nogo}, we discuss what is known about the consistency of generic hugeness by itself and present a corollary of Theorem \ref{main} showing that the usual forcing strategies cannot produce models where $\omega_1$ is generically huge with target $\omega_2$ by a strongly $\omega_2$-c.c.\ poset.
Our notations and terminology are standard.  We assume the reader is familiar with the basics of forcing and elementary embeddings.

\section{Generic huge embeddings and approximation}
\label{eva}

The relevance of the strong $\kappa$-c.c.\ is its connection to the approximation property of Hamkins \cite{hamkins}.  Suppose $\calF \subseteq \p(\lambda)$.  We say that a set $X \subseteq \lambda$ is \emph{approximated by $\calF$} when $X \cap z \in \calF$ for all $z \in \calF$.  If $V \subseteq W$ are models of set theory, then we say that the pair $(V,W)$ satisfies the \emph{$\kappa$-approximation property} for a $V$-cardinal $\kappa$ when for all $\lambda \in V$ and all $X \subseteq \lambda$ in $W$, if $X$ is approximated by $\p_\kappa(\lambda)^V$, then $X \in V$.  We say that a forcing $\bbP$ \emph{has the $\kappa$-approximation property} when the $\kappa$-approximation property is forced to hold of the pair $(V,V[G])$.
The following result appears as Lemma 1.5 and Note 1.11 in \cite{usubaapprox}:
 
\begin{theorem}[Usuba]
If $\bbP$ is a nontrivial $\kappa$-c.c.\ forcing and $\dot{\bbQ}$ is a $\bbP$-name for a $\kappa$-closed forcing, then $\bbP * \dot{\bbQ}$ has the $\kappa$-approximation property if and only if $\bbP$ has the strong $\kappa$-c.c.
\end{theorem}

Theorem \ref{main} will follow from the more general lemma below.

\begin{lemma}
\label{killa}
The following hypotheses are jointly inconsistent:
\begin{enumerate}
\item $\kappa_0\leq\kappa_1$ and $\lambda_0\leq\lambda_1$ are regular cardinals.
\item\label{genhuge} $\bbP$ is a nontrivial strongly $\lambda_0$-c.c.\ poset that forces an elementary embedding $j : V \to M \subseteq V[G]$ with $j(\kappa_0) = \lambda_0$, $j(\kappa_1) = \lambda_1$, $\p(\lambda_1)^V \subseteq M$, and $M^{<\lambda_0} \cap V[G] \subseteq M$.
\item\label{genmeas} $\kappa_1^+$ is $\bbQ$-generically measurable for a $\kappa_0$-closed $\bbQ$.
\end{enumerate}
\end{lemma}

\begin{proof}
We will need a first-order version of (\ref{genmeas}) that can be carried through the embedding of (\ref{genhuge}).  Replace it by the (possibly weaker) hypothesis that $\bbQ$ is a $\kappa_0$-closed poset and for some $\theta \gg \lambda_1$, $\bbQ$ forces an elementary embedding $j : H_\theta^V \to N$ with critical point $\kappa_1^+$, where $N \in V^{\bbQ}$ is a transitive set. 

\begin{claim}
\label{gch1}
$\kappa_1^{<\kappa_0} = \kappa_1$.
\end{claim}
\begin{proof}
Let $G \subseteq \bbQ$ be generic over $V$, and let $j : H_\theta^V \to N$ be an elementary embedding with critical point $\kappa_1^+$, where $N \in V[G]$ is a transitive set.  By ${<}\kappa_0$-distributivity, $\p_{\kappa_0}(\kappa_1)^{N} \subseteq\p_{\kappa_0}(\kappa_1)^{V}$, so the cardinality of $\p_{\kappa_0}(\kappa_1)^V$ must be below the critical point of $j$.
\end{proof}

\begin{claim}
\label{gch2}
$\lambda_1^{<\lambda_0} = \lambda_1$.
\end{claim}
\begin{proof}
Let $G \subseteq \bbP$ be generic over $V$, and let $j : V \to M$ be as hypothesized in (\ref{genhuge}).  
By the closure of $M$, $\p_{\lambda_0}(\lambda_1)^M = \p_{\lambda_0}(\lambda_1)^{V[G]}$.
By elementarity and Claim \ref{gch1}, $M \models \lambda_1^{<\lambda_0} = \lambda_1$.  Thus $M$ has a surjection $f : \lambda_1 \to \p_{\lambda_0}(\lambda_1)^{V[G]} \supseteq \p_{\lambda_0}(\lambda_1)^V$.  If $\lambda_1^{<\lambda_0} > \lambda_1$ in $V$, then $f$ would witnesses a collapse of $\lambda_1^+$, contrary to the $\lambda_0$-c.c.
\end{proof}

Now let $\calF = \p_{\lambda_0}(\lambda_1)^V$.  Let $j : V \to M \subseteq V[G]$ be as in hypothesis (\ref{genhuge}).  Claim \ref{gch2} implies that $\calF$ is coded by a single subset of $\lambda_1$ in $V$, so $\calF \in M$.  In $M$, let $\calA$ be the collection of subsets of $\lambda_1$ that are approximated by $\calF$.  Since $\p(\lambda_1)^V \subseteq M$, it is clear that $\p(\lambda_1)^V \subseteq \calA$.

For each $\alpha<\lambda_1^+$, there exists an $X \in \calA \cap V$ that codes a surjection from $\lambda_1$ to $\alpha$ in some canonical way.  Working in $M$, choose for each $\alpha<\lambda_1^+$ an $X_\alpha\in\calA$ that codes a surjection from $\lambda_1$ to $\alpha$.

By elementarity, $\lambda_1^+$ is $j(\bbQ)$-generically measurable in $M$, witnessed by generic embeddings with domain $H^M_{j(\theta)}$.  By the closure of $M$, $j(\bbQ)$ is $\lambda_0$-closed in $V[G]$.  Let $H \subseteq j(\bbQ)$ be generic over $V[G]$.  Let $i : H^M_{j(\theta)} \to N \in M[H]\subseteq V[G][H]$ be given by the $j(\bbQ)$-generic measurability of $\lambda_1^+$ in $M$, with $\crit(i) = \delta = \lambda_1^+$.

Let $\la X'_\alpha : \alpha < i(\delta) \ra = i(\la X_\alpha : \alpha < \delta \ra)$.  By elementarity, $X'_\delta$ is approximated by $i(\calF) = \calF$.
Since $\bbP * j(\dot\bbQ)$ is a nontrivial strongly $\lambda_0$-c.c.\ forcing followed by a $\lambda_0$-closed forcing, it has the $\lambda_0$-approximation property by Usuba's Theorem.  
Therefore, $X'_\delta \in V$.
But this is a contradiction, since $X'_\delta$ codes a surjection from $\lambda_1$ to $(\lambda_1^+)^V$.
\end{proof}

Let us now complete the proof of Theorem \ref{main}.  
Suppose $n\geq 1$, $\kappa<\lambda$, $\bbP$ is strongly $\lambda$-c.c., and $\bbP$ forces an embedding $j : V \to M \subseteq V[G]$ such that $j(\kappa) = \lambda$ and $M$ is closed under $j^n(\kappa)$-sequences from $V[G]$.
By the $\lambda$-c.c.\ of $\bbP$ and the $\lambda$-closure of $M$, $(\lambda^+)^M = (\lambda^+)^V$.  Suppose inductively that $i<n$ and $(\lambda^{+i})^M = (\lambda^{+i})^V \leq j^{i+1}(\kappa)$.  Again, by the chain condition and the $j^{i+1}(\kappa)$-closure of $M$, $(\lambda^{+i+1})^M = (\lambda^{+i+1})^V$.  
Since $\kappa^{+i}<\lambda^{+i} = j(\kappa^{+i})$, $j(\lambda^{+i})$ must be an $M$-cardinal greater than $\lambda^{+i}$, so $\lambda^{+i+1} \leq j(\lambda^{+i})$.  By elementarity applied to the induction hypothesis, $j(\lambda^{+i}) \leq j^{i+2}(\kappa)$.  Thus the induction hypothesis carries through up to $n$.  Now suppose $0<m\leq n$ and set $\kappa_0=\kappa$, $\lambda_0 = \lambda$, $\kappa_1 = \kappa^{+m-1}$, and $\lambda_1 = \lambda_0^{+m-1}$.  Then we have $j(\kappa_0)=\lambda_0$ and $j(\kappa_1) = \lambda_1 \leq j^n(\kappa)$. If $\kappa^{+m}$ is also generically measurable by a $\kappa$-closed forcing, then this assignment of variables satisfies the hypotheses of the lemma, which we have shown to be inconsistent.

\begin{remark}
Suppose $\omega_1$ is $\bbP$-generically almost-huge and $\omega_2$ is $\bbQ$-generically measurable, where $\bbP$ is strongly $\omega_2$-c.c.\ and $\bbQ$ is countably closed.  This holds, for example, in Foreman's model \cite{foremanmore}.
Let $j : V \to M$ be an embedding witnessing the $\bbP$-generic almost-hugeness of $\omega_1$.  Put $\kappa_0=\kappa_1=\omega_1$ and $\lambda_0=\lambda_1=\omega_2$.  The only hypothesis of Lemma \ref{killa} that fails is $\p(\omega_2)^V \subseteq M$.
\end{remark}

\section{On the consistency of generic hugeness}
\label{nogo}

It is not known whether any successor cardinal can be minimally generically huge.  Moreover, it is not known whether $\omega_1$ can be $\bbP$-generically huge with target $\omega_2$ for an $\omega_2$-c.c.\ forcing $\bbP$.  But we do not think that Theorem \ref{main} is evidence that this hypothesis by itself is inconsistent, since there are other versions of generic hugeness for $\omega_1$ that satisfy the hypothesis of Theorem \ref{main} and are known to be consistent relative to huge cardinals.  Magidor \cite{MR526312} showed that if there is a huge cardinal, then in a generic extension, $\omega_1$ is $\bbP$-generically huge with target $\omega_3$, where $\bbP$ is strongly $\omega_3$-c.c.  Shioya \cite{MR4159767} observed that if $\kappa$ is huge with target $\lambda$, then Magidor's result can be obtained from a two-step iteration of Easton collapses, $\bbE(\omega,\kappa) * \dot\bbE(\kappa^+,\lambda)$.  An easier argument shows that after the first step of the iteration, or even in the extension by the Levy collapse $\col(\omega,{<}\kappa)$, $\omega_1$ is $\bbP$-generically huge with target $\lambda$ by a strongly $\lambda$-c.c.\ forcing $\bbP$.

Theorem \ref{main} shows that in these models, $\omega_2$ is not $\bbQ$-generically measurable for a countably closed $\bbQ$.
It also shows that if it is consistent for $\omega_1$ to be generically huge with target $\omega_2$ by a strongly $\omega_2$-c.c.\ forcing, then this cannot be demonstrated by a standard method resembling Magidor's:

\begin{corollary}
Suppose $\kappa$ is a huge cardinal with target $\lambda$.  Suppose $\bbP$ is such that:
\begin{enumerate}
\item $\bbP$ is $\lambda$-c.c.\ and contained in $V_\lambda$.
\item $\bbP$ preserves $\kappa$ and collapses $\lambda$ to become $\kappa^+$.
\item\label{decomp} For all sufficiently large $\alpha<\lambda$ (for example, all Mahlo $\alpha$ beyond a certain point), $\bbP \cong (\bbP \cap V_\alpha) * \dot\bbQ_\alpha$, where $\dot\bbQ_\alpha$ is forced to be $\kappa$-closed.
\end{enumerate}
Then in any generic extension by $\bbP$, $\kappa$ is not generically huge with target $\lambda$ by a strongly $\lambda$-c.c.\ forcing.

Furthermore, suppose $\lambda$ is supercompact in $V$, and (\ref{decomp}) is strengthened to:
\begin{enumerate}
\setcounter{enumi}{3}
\item\label{sc} For all sufficiently large $\alpha<\beta<\lambda$, $\bbP \cong (\bbP \cap V_\alpha) * \dot\col(\kappa,\beta)* \dot\bbQ_{\alpha,\beta}$, where $\dot\bbQ_ {\alpha,\beta}$ is forced to be $\kappa$-closed.
\end{enumerate}
Then $\kappa$ is not generically huge with target $\lambda$ by a strongly $\lambda$-c.c.\ forcing in any $\lambda$-directed-closed forcing extension of $V^{\bbP}$.
\end{corollary}

\begin{proof}
Let $j : V \to M$ witness that $\kappa$ is huge with target $\lambda$.  By elementarity and the fact that $\p(\lambda) \subseteq M$, $\lambda$ is measurable in $V$.  
Let $\calU$ be a normal ultrafilter on $\lambda$, and let $i : V \to N$ be the ultrapower embedding.

Since the decomposition of (\ref{decomp}) holds for all ``sufficiently large'' $\alpha$, $N \models i(\bbP) \cong \bbP * \dot\bbQ$, where $\dot\bbQ$ is forced to be $\kappa$-closed.  By the closure of $N$, $V$ also believes that $\dot\bbQ$ is forced by $\bbP$ to be $\kappa$-closed.  Thus if we take $G \subseteq \bbP$ generic over $V$, then the embedding $i$ can be lifted by forcing with $\bbQ$.  This means that in $V[G]$, $\lambda$ is $\bbQ$-generically measurable, $\bbQ$ is $\kappa$-closed, and $\lambda = \kappa^+$.  Theorem \ref{main} implies that in $V[G]$, $\kappa$ cannot be generically huge with target $\lambda$ by a strongly $\lambda$-c.c.\ forcing.

For the final claim, suppose $\lambda$ is supercompact in $V$, and let $\dot\bbR$ be a $\bbP$-name for a $\lambda$-directed-closed forcing.
Let $\gamma$ be such that $\Vdash_{\bbP} |\dot\bbR| \leq \gamma$. By \cite[Theorem 14.1]{MR2768691}, $\col(\kappa,\gamma) \cong \col(\kappa,\gamma) \times \bbR$ in $V^{\bbP}$.
Let $i : V \to N$ be an elementary embedding such that $\crit(i) = \lambda$, $i(\lambda) > \gamma$, and $N^\gamma \subseteq N$.
By applying (\ref{sc}) in $N$, there is in $N$ a complete embedding of $\bbP * \dot\bbR$ into $i(\bbP)$, such that the quotient forcing is equivalent to something of the form $\col(\kappa,\gamma)*\dot\bbQ_{\lambda,\gamma}$, where $\dot\bbQ_{\lambda,\gamma}$ is forced to be $\kappa$-closed in $N^{\bbP*\dot\bbR * \dot\col(\kappa,\gamma)}$.  By the closure of $N$, the quotient is forced to be $\kappa$-closed in $V^{\bbP*\dot\bbR}$.

Let $G * H \subseteq \bbP*\dot\bbR$ be generic.  Further $\kappa$-closed forcing yields a generic $G' \subseteq i(\bbP)$ that projects to $G*H$.  We can lift the embedding to $i : V[G] \to N[G']$.  By elementarity, $i(\bbR)$ is $i(\lambda)$-directed-closed in $N[G']$.  Thus $i[H]$ has a lower bound $r \in i(\bbR)$.  By the closure of $N$, $i(\bbR)$ is at least $\kappa$-closed in $V[G']$.   Forcing below $r$ yields a generic $H' \subseteq i(\bbR)$ and a lifted embedding $i : V[G*H] \to N[G'*H']$.  
Hence in $V[G*H]$, $\lambda$ is generically measurable via a $\kappa$-closed forcing.  Theorem \ref{main} implies that $\kappa$ cannot be generically huge with target $\lambda$ by a strongly $\lambda$-c.c.\ forcing.
\end{proof}

\bibliographystyle{amsplain}
\bibliography{genhugebib}
\end{document}